\documentclass[a4paper, 12pt]{amsart}
\usepackage{txfonts}
\usepackage{amssymb}
\usepackage{mathrsfs}
\usepackage{latexsym}
\usepackage{cases}
\usepackage{amscd}
\usepackage{amsfonts}
\usepackage{amsmath}
\usepackage[all]{xy}
\usepackage{xy}
\usepackage{enumerate}

\newtheorem{theorem}{Theorem}[section]
\newtheorem{lemma}[theorem]{Lemma}

\newtheorem{proposition}[theorem]{Proposition}
\newtheorem{corollary}[theorem]{Corollary}

\theoremstyle{definition}

\newtheorem{example}[theorem]{Example}
\newtheorem{question}{Question}

\theoremstyle{remark}
\newtheorem{remark}[theorem]{Remark}

\numberwithin{equation}{section}

\newcommand\C{\mathbb{C}}

\newcommand\Z{\mathbb{Z}}
\newcommand\Ha{\mathbb{H}}

\begin{document}

\def\temptablewidth{0.5\textwidth}

\title[Connected sums of almost complex manifolds]{Connected sums of almost complex manifolds}

\author[Huijun Yang]{Huijun Yang}
\address{School of Mathematics and Statistics, Henan University, Kaifeng 475004, Henan, China}
\email{yhj@amss.ac.cn}

\begin{abstract}
 In this paper, firstly, for any $4n$-dimensional almost complex manifolds $M_{i}, ~1\le i \le \alpha$, we prove that $\left(\sharp_{i=1}^{\alpha} M_{i}\right) \sharp (\alpha{-}1) \C P^{2n}$ must admits an almost complex structure, where $\alpha$ is a positive integer. Secondly, for a $2n$-dimensional almost complex manifold $M$, we get that $M\sharp \overline{\C P^{n}}$ also admits an almost complex structure. At last, as an application, we obtain that $\alpha\C P^{2n}\sharp \beta\overline{\C P^{2n}}$ admits an almost complex structure if and only if $\alpha$ is odd.
\end{abstract}


\maketitle


\section{Introduction} 
\label{s:intro}
Through out this paper, all the manifolds are closed, connected, oriented and smooth. 

Let $M$ be a manifold. Denote by $TM$ the tangent bundle of $M$ and $\epsilon^{k}$ the $k$-dimensional trivial real vector bundle over $M$. We say that $M$ is an almost (resp. a stable almost) complex manifolds, i.e., admits an almost (resp. a stable almost) complex structure, if there exists a endomorphism $J\colon TM\rightarrow TM$ (resp. $J\colon TM\oplus \epsilon^{k}\rightarrow TM\oplus \epsilon^{k}$ for some $k$) such that $J^{2}=-1$. If $M$ admits an almost complex structure, it follows from the definition that $M$ must be orientable and with even dimension.

Suppose that $M$ admits an almost complex structure with $\dim M =4n$. Denote by $\chi(M)$ and $\tau(M)$ the Euler characteristic and signature of $M$ respectively.  Then Hirzebruch \cite[p. 777]{hi87} tell us that we must have
\begin{equation}\label{eq:hir}
\chi(M) \equiv (-1)^{n}\tau(M)  \mod 4.
\end{equation}

\begin{remark}
For $n > 1$, the congruence \eqref{eq:hir} can also be deduced from Tang and Zhang \cite[Corollary 3.8]{tz02}.
\end{remark}
\begin{example}\label{exam:cp2n}
Denote by $\C P^{n}$ the $n$-dimensional complex projective space with the natural orientation induced by the complex structure, and $\overline{\C P^{n}}$ the same manifold with the opposite orientation. We know that $\chi(\C P^{2n}) = \chi(\overline{\C P^{2n}}) = 2n{+}1$, $\tau(\C P^{2n}) = 1$ and $\tau(\overline{\C P^{2n}}) = -1$. Therefore, it follows from the congruence \eqref{eq:hir} that $\overline{\C P^{2n}}$ does not admit an almost complex structure for $n \ge 1$ (see also for instance Tang and Zhang \cite[Corollary 3.9]{tz02}).
\end{example}

Let $\alpha$ be a positive integer, and $M_{i},~1\le  i \le \alpha,$ be $4n$-dimensional almost complex manifolds (may be different). Denote by $\alpha M_{1}$ the connected sum of $\alpha$ copies of $M_{1}$, and 
$\sharp_{i=1}^{\alpha} M_{i}$
the connected sum of $M_{i}, ~1\le i \le \alpha$. 
Then we have 
\begin{align}
\chi(\sharp_{i=1}^{\alpha} M_{i}) &= \Sigma_{i=1}^{\alpha}\chi(M_{i}) - 2(\alpha-1), \label{eq:chi}\\
\tau(\sharp_{i=1}^{\alpha} M_{i}) &= \Sigma_{i=1}^{\alpha}\tau(M_{i}). \label{eq:tau}
\end{align}
Hence it follows immediately from the congruence \eqref{eq:hir} that
\begin{proposition}\label{prop:even}
Let $M_{i},~1\le  i \le \alpha,$ be $4n$-dimensional almost complex manifolds. If $\alpha$ is even, then $\sharp_{i=1}^{\alpha} M_{i}$ does not admit an almost complex structure.
\end{proposition}

\begin{remark}\label{rem:cpodd}
In Proposition \ref{prop:even}, the conditions $M_{i},~1\le i \le \alpha$ are almost complex manifolds and $\dim M_{i} =4n$ are necessary because of the following facts.
\begin{itemize}
\item[(1)] Denote by $S^{n}$ the $n$-dimensional standard sphere. It can be deduced easily from M\"{u}ller and Geiges \cite[Theorem 4]{mg00} that $S^{4}\times S^{4}\sharp \overline{\C P^{4}}$ do admits an almost complex structure. However, it follows from Sutherland \cite[Theorem 3.1]{su65} (or Yang \cite[Theorems 1 and 2]{ya12} or Datta and Subramanian \cite[Theorem 1]{ds90}) and Example \ref{exam:cp2n} that both $S^{4}\times S^{4}$ and $\overline{\C P^{4}}$ do not admit an almost complex structure. 
\item[(2)] We know that $\alpha \C P^{2n+1}$ is diffeomorphism to $\C P^{2n+1} \sharp (\alpha{-}1)\overline{\C P^{2n+1}}$ which is a blow up of $\C P^{2n+1}$ in $(\alpha{-}1)$ point. Therefore, $\alpha\C P^{2n+1}$ must be K\"{a}hler, hence admits almost complex structure for any positive integer $\alpha$.
\end{itemize}
\end{remark}

Now, based on the facts of Proposition \ref{prop:even} and Remark \ref{rem:cpodd}, it is natural to ask the following naive question:

\begin{question}\label{que:oddacs}
For any odd positive integer $\alpha$ and any $4n$-dimensional almost complex manifolds $M_{i},~1\le  i \le \alpha$, does $\sharp_{i=1}^{\alpha}M_{i}$ admits an almost complex structure?
\end{question}

For example,  it is known that $S^{6}$ admits an almost complex structure, hence $S^{6} \times S^{6}$ admits an almost complex structure. Therefore, It can be deduced from Yang \cite[Theorems 1 and 2]{ya12} that $\alpha S^{6} \times S^{6}$ admits an almost complex structure if and only if $\alpha$ is odd (the proof is left to the reader).

However, the answer to Question \ref{que:oddacs} is negative in generally. 
\begin{example}
We can deduced from Yang \cite[Theorem 2]{ya12} that $\alpha S^{10} \times S^{10}$ admits an almost complex structure if and only if $\alpha \equiv -1 \bmod 1152$. Therefore let $M = 1151 S^{10} \times S^{10}$, it must admits an almost complex structure. Moreover, $\alpha M$ admits an almost complex structure if and only if $\alpha \equiv 1 \bmod 1152$.
\end{example}
In fact, if we set 
$\mathscr{M}_{4m}$ be the set of $(4m{-}1)$-connected $8m$-dimensional smooth manifolds for which admit almost complex structure.
Then it can be deduced from Yang \cite[Lemma 1, Theorem 2]{ya12} that
\begin{proposition}\label{prop:m4m}
For any $M_{i}\in \mathscr{M}_{4m}, 1\le i \le \alpha$, $\sharp_{i=1}^{\alpha}M_{i}$ admits an almost complex structure if and only if $\alpha =1$.
\end{proposition}
\begin{proof}
For a $(4m{-}1)$-connected $8m$-manifold $M$, a necessary condition for $M$ to admits an almost complex structure is (cf. Yang \cite[Theorem 2]{ya12})
\begin{equation}\label{eq:m4m}
4p_{2m}(M)-p^{2}_{m}(M)=8\chi(M),
\end{equation}
where $p_{i}(M)$ is the $i$-th Pontrjagin class of $M$. 
Note that for these manifolds, we must have (cf. Yang \cite[Lemma 1]{ya12})
\begin{align}
p_{2m}(\sharp_{i=1}^{\alpha}M_{i}) &= \Sigma_{i=1}^{\alpha}p_{2m}(M_{i}), \label{eq:p2m}\\
p_{m}^{2}(\sharp_{i=1}^{\alpha}M_{i}) &= \Sigma_{i=1}^{\alpha}p_{m}^{2}(M_{i}). \label{eq:pm2}
\end{align}
Then the facts of this proposition follows easily from the necessary condition \eqref{eq:m4m} and the equations \eqref{eq:chi}, \eqref{eq:p2m} and \eqref{eq:pm2}.
\end{proof}

\begin{remark}
It follows from M\"{u}ller and Geiges \cite[Proposition 6]{mg00} that
$$\Ha P^{2}\sharp\Ha P^{2}\sharp S^{4}\times S^{4} \in \mathscr{M}_{4},$$
where $\Ha P^{2}$ is the quaternionic projective plane. Hence $\mathscr{M}_{4}\neq \emptyset$.
\end{remark}

Even though the answer to Question \ref{que:oddacs} is negative in generally, it may positive if some $M_{i}$ in Question \ref{que:oddacs} are fixed into some particular almost complex manifolds. 
In this paper, our main results are stated as:

\begin{theorem}\label{thm:main}
For any positive integer $\alpha$ and $4n$-dimensional almost complex manifolds $M_{i},~1\le i\le \alpha$, the connected sum 
\begin{equation*}
\left(\sharp_{i=1}^{\alpha}M_{i}\right)\sharp (\alpha{-}1)\C P^{2n}
\end{equation*}
must admits an almost complex structure.
\end{theorem}

\begin{remark}
For $n \ge 2$ and $2n$-dimensional almost complex manifolds $M_{i},~1 \le i \le \alpha$, Geiges has proved in \cite[Lemma 2]{ge01} that $\left(\sharp_{i=1}^{\alpha}M_{i}\right)\sharp (\alpha{-}1)S^{2} \times S^{2n-2}$ must admits an almost complex structure. However, $S^{2}\times S^{2n-2}$ does not admits an almost complex structure whence $n \ge 4$ (cf. Sutherland \cite[Theorem 3.1]{su65} or Datta and Subramanian \cite[Theorem 1]{ds90}).
\end{remark}

Therefore, it follows immediately from Proposition \ref{prop:even} and Theorem \ref{thm:main} that
\begin{corollary}\label{coro:cpodd}
$\alpha\C P^{2n}$ admits an almost complex structure if and only if $\alpha$ is odd.
\end{corollary}

\begin{remark}
This fact has been got by Goertsches and Konstantis \cite{gk17}.
\end{remark}

The facts of theorem \ref{thm:main} lead us to consider more about the connected sum of almost complex manifolds with complex projective spaces.
\begin{theorem}\label{thm:mcpn}
Let $M$ be a $2n$-dimensional almost complex manifolds. Then $M\sharp \overline{\C P^{n}}$ must admits an almost complex structure.
\end{theorem}

\begin{remark}
Even if $M\sharp \overline{\C P^{n}}$ admits an almost complex structure, $M$ may not admits an almost complex structure. For example: \\
(1) we known that $S^{2n}$ does not admit an almost complex structure for $k\ge 4$. However, for odd $n$, $S^{2n}\sharp \overline{\C P^{n}} = \overline{\C P^{n}}$ is diffeomorphic to $\C P^{n}$ which admits an almost complex structure.\\
(2) although $S^{4} \times S^{4}$ does not admit an almost complex structure, Remark \ref{rem:cpodd} tell us that $S^{4} \times S^{4} \sharp \overline{\C P^{4}}$ admits an almost complex structure.
\end{remark}

As an application, combing the congruence \eqref{eq:hir} with Theorem \ref{thm:mcpn} and Corollary \ref{coro:cpodd}, we can get that 
\begin{corollary}
$\alpha\C P^{2n} \sharp \beta \overline{\C P^{2n}}$ admits an almost complex structure if and only if $\alpha$ is odd.
\end{corollary}
The proof of this corollary is left to the reader.
\begin{remark}
For $n=1$ and $2$, the facts of this corollary have been obtained by Audin \cite{au91} and M\"uller and Geiges \cite{mg00} respectively. 
\end{remark}
\begin{remark}
The necessary and sufficient conditions for the existence of almost complex structure on $\alpha\C P^{n} \sharp \beta \overline{\C P^{n}}$ are investigate by Sato and Suzuki \cite[p. 102, Proposition]{ss74}. Unfortunately, their results are not correct.
\end{remark}

The proof of Theorems \ref{thm:main} and \ref{thm:mcpn} will be given in section 2.

\section{Proof of Theorems \ref{thm:main} and \ref{thm:mcpn}}  \label{s:proof}
In order to prove our main results Theorems \ref{thm:main} and \ref{thm:mcpn}, we need some preliminaries. 

Let $M$ be a $2n$-dimensional oriented manifold with tangent bundle $TM$. 
For a complex vector bundle $\eta$ over $M$, denote by 
$\bar{\eta}$ and $\eta_{_{R}}$ the conjugate and real reduction bundle of $\eta$ respectively. It is known that 
\begin{equation*}
\eta_{_{R}} = \bar{\eta}_{_{R}}.
\end{equation*}
If the complex vector bundle $\eta$ satisfies that $\eta_{_{R}}$ is isomorphic (resp. stably isomorphic) to $TM$, it follows from the definition of almost (resp. stable almost) complex structure that $\eta$ determines an almost (resp. a stable almost) complex structure on $M$, and we denote it as $J_{\eta}$ (resp. $\tilde{J}_{\eta}$).  

For the complex projective space $\C P^{n}$, denote by $\gamma$ the canonical complex line bundle over $\C P^{n}$, $TP$ the real tangent bundle of $\C P^{n}$ and $T\overline{P}$ the real tangent bundle of $\overline{\C P^{n}}$.

Moreover, we should use the results and conventions of Kahn \cite{ka69}.
Let $J$ be an almost complex structure on $M {-} D^{2n}$ for some embedded disc $D^{2n}$. Denote by
\begin{equation*}
\mathfrak{o}(M, J) \in H^{2n}(M; \pi_{2n-1}(SO(2n)/U(n)))
\end{equation*} 
the obstruction to extending $J$ as an almost complex structure over $M$, and set
\begin{equation*}
\mathfrak{o}[M, J] = \langle \mathfrak{o}(M, J), [M]\rangle
\end{equation*}
where $[M]$ is the fundamental class of $M$ and $\langle~,~\rangle$ is the Kronecker product.

Then we have the following statements from Kahn \cite{ka69} (cf. Geiges \cite{ge01}):
\begin{lemma}[Kahn]\label{lem:os}
$\mathfrak{o}[S^{2n}, J]$ is independent of $J$ and will be written as $\mathfrak{o}[S^{2n}]$.
\end{lemma}

\begin{lemma}[Kahn]\label{lem:osum}
Almost complex structure $J$ on $M{-}D^{2n}$ and $J^{\prime}$ on $M^{\prime}{-}D^{2n}$ give rise to a natural almost complex structure $J+J^{\prime}$ on $M\sharp M^{\prime}{-}D^{2n}$ (which coincides with $J$ resp. $J^{\prime}$ along the $(2n{-}1)$-skeleton of $M\sharp M^{\prime}$) such that 
\begin{equation*}
\mathfrak{o}[M\sharp M^{\prime}, J+J^{\prime}] = \mathfrak{o}[M, J] + \mathfrak{o}[M^{\prime}, J^{\prime}] - \mathfrak{o}[S^{2n}].
\end{equation*}
\end{lemma}

\begin{lemma}[Kahn]\label{lem:och}
Let $J$ be an almost complex structure on $M{-}D^{2n}$ that extends over $M$ as a stable almost complex structure $\tilde{J}$. Then
\begin{equation*}
\mathfrak{o}[M, J] = \frac{1}{2}\left(\chi(M)-c_{n}[\tilde{J}]\right)\mathfrak{o}[S^{2n}],
\end{equation*}
where $c_{n}[\tilde{J}]$ is the top Chern number of $\tilde{J}$.
\end{lemma}

Now we are in position to prove the Theorems \ref{thm:main} and \ref{thm:mcpn}.

\begin{proof}[Proof of Theorem \ref{thm:main}]
Firstly, let us consider the stable almost complex structure on $\C P^{2n}$. Let $\eta = (2n{-}1)\gamma + 2\bar{\gamma}$. It is known that 
\begin{equation*}
\eta_{_{R}} = (2n+1) \gamma_{_{R}} = (2n+1)\bar{\gamma}_{_{R}}
\end{equation*}
is stably isomorphic to $TP$. Hence $\eta$ determines a stable almost complex structure $\tilde{J}_{\eta}$ on $\C P^{2n}$ and the total Chern class of $\tilde{J}_{\eta}$ is 
\begin{equation*}
c(\tilde{J}_{\eta}) = c(\eta) = c(\gamma)^{2n-1}c(\bar{\gamma})^{2} = (1+x)^{2n-1}(1-x)^{2},
\end{equation*}
where $x \in H^{2}(\C P^{2n}; \Z)$ is the first Chern class of $\gamma$. Therefore,
\begin{equation*}
c_{2n}[\tilde{J_{\eta}}] = c_{2n}[\eta] = 2n-3.
\end{equation*}
Since $\C P^{2n}{-}D^{4n}$ is homotopic equivalent to $\C P^{2n-1}$ and the coefficient groups $\pi_{r}(SO(4n)/U(2n))$ for the obstructions to an almost complex structure are stable for $r < 4n{-}1$ (cf. Massey \cite{ma61}), the stable almost complex structure $\tilde{J}_{\eta}$ induces an almost complex structure $J_{\eta}$ on $\C P^{2n}{-}D^{4n}$.
Then it follows from Lemma \ref{lem:och} that
\begin{equation*}
\mathfrak{o}[\C P^{2n}, J_{\eta}] = 2 \mathfrak{o}[S^{4n}].
\end{equation*}

Now denote by $J_{i}$ the given almost complex structures on $M_{i}, ~1 \le i \le \alpha$. It is clearly that
\begin{equation*}
\mathfrak{o}[M_{i}, J_{i}] = 0.
\end{equation*}
Consequently, Lemma \ref{lem:osum} implies that 
\begin{align*}
& \mathfrak{o}[\sharp_{i=1}^{\alpha}M_{i}\sharp (\alpha{-}1) \C P^{2n}, ~\Sigma_{i=1}^{\alpha}J_{i}{+}(\alpha{-}1)J_{\eta}] \\
= ~& \Sigma_{i=1}^{\alpha}\mathfrak{o}[M_{i},~J_{i}] + (\alpha{-}1)\mathfrak{o}[\C P^{2n}, J_{\eta}] - (2\alpha {-} 2)\mathfrak{o}[S^{4n}]\\
=~& 0.
\end{align*}
This completes the proof.
\end{proof}

\begin{proof}[Proof of Theorem \ref{thm:mcpn}]
Firstly, let us consider the stable almost complex structures on $\C P^{n}$. In the stable range, it is obviously that the tangent bundle $T\overline{P}$ of $\overline{\C P^{n}}$ is stably isomorphic to the tangent bundle $TP$ of $\C P^{n}$. Thus, let $\eta = n\gamma + \bar{\gamma}$, it is follows that 
\begin{equation*}
\eta_{_{R}} = (n+1)\gamma_{_{R}} = (n+1)\bar{\gamma}_{_{R}}
\end{equation*}
is stably isomorphic to $T\overline{P}$. Hence $\eta$ determines a stable almost complex structure $\tilde{J}_{\eta}$ on $\overline{\C P^{n}}$ and the total Chern class of $\tilde{J}_{\eta}$ is 
\begin{equation*}
c(\tilde{J}_{\eta}) = c(\eta) = c(\gamma)^{n}c(\bar{\gamma})^{} = (1+x)^{n}(1-x).
\end{equation*}
Therefore,
\begin{equation*}
c_{n}[\tilde{J_{\eta}}] = c_{n}[\eta] = \langle c_{n}(\eta), [\overline{\C P^{n}}] \rangle = n-1.
\end{equation*}

Then, as in the proof of the Theorem \ref{thm:main}, the stable almost complex structure $\tilde{J}_{\eta}$ induces an almost complex structure $J_{\eta}$ on $\overline{\C P^{n}}{-}D^{2n}$.
Hence it follows from Lemma \ref{lem:och} that
\begin{equation*}
\mathfrak{o}[\overline{\C P^{n}}, J_{\eta}] =  \mathfrak{o}[S^{2n}].
\end{equation*}

Now denote by $J$ the given almost complex structures on $M$. It is clearly that
\begin{equation*}
\mathfrak{o}[M, J] = 0.
\end{equation*}
Consequently, Lemma \ref{lem:osum} implies that 
\begin{align*}
\mathfrak{o}[M \sharp \overline{\C P^{n},} ~J{+}J_{\eta}]
= \mathfrak{o}[M,~J] + \mathfrak{o}[\overline{\C P^{n}}, J_{\eta}] - \mathfrak{o}[S^{2n}]
= 0.
\end{align*}
This completes the proof.
\end{proof}



\noindent
\thanks{\textbf{Acknowledgment.} The author would like to thank the University of Melbourne where parts of this work to be carried out and also Diarmuid Crowley for his hospitality. The author is partially supported by the China Scholarship Council (File No. 201708410052).}








%
%


%


\end{document}